\documentclass[11pt]{amsart}

\usepackage[margin=1.25in]{geometry} 
\usepackage{amsmath,amsthm,amssymb,bbm,bm, breqn}
\usepackage{graphicx, xcolor}
\usepackage{hyperref}
\usepackage{amssymb,amsfonts}
\usepackage[all,arc]{xy}
\usepackage{enumerate}
\usepackage{mathrsfs}
\usepackage{xcolor}
\usepackage{tikz}
\usepackage{enumerate, enumitem}

\newcommand{\N}{\mathbb{N}}

\newcommand{\Fc}{\mathcal{F}}

\newcommand{\I}{\mathbbm{1}}

\newcommand{\Poi}{\mathrm{Poisson}}
\newcommand{\fp}{\mathrm{fp}}
\newcommand{\tc}{\mathrm{tc}}
\newcommand{\PF}{\mathrm{PF}}
\newcommand{\PC}{\mathrm{PC}}

\newcommand{\Address}{{
\bigskip
\footnotesize

\textsc{Department of Mathematics, University of Southern California, Los Angeles, CA}\par\nopagebreak
\textit{E-mail address}: \texttt{paguyo@usc.edu}
}}

\def\bal#1\eal{\begin{align*}#1\end{align*}}

\newtheorem{theorem}{Theorem}[section]
\newtheorem{lemma}[theorem]{Lemma}
\newtheorem{proposition}[theorem]{Proposition}

\title[Cycle structure of random parking functions]{Cycle structure of random parking functions}
\author{J. E. Paguyo}

\date{\today}


\begin{document}

%%%%%%%%%%%%%%%%%%%%%%%%%%%%%%%%%%%%%%%%%%%%%%%%%%%%%%%%%
%%% ABSTRACT %%%

\begin{abstract} 
We initiate the study of the cycle structure of uniformly random parking functions. Using the combinatorics of parking completions, we compute the asymptotic expected value of the number of cycles of any fixed length. 
We obtain an upper bound on the total variation distance between the joint distribution of cycle counts and independent Poisson random variables using a multivariate version of Stein's method via exchangeable pairs. 
Under a mild condition, the process of cycle counts converges in distribution to a process of independent Poisson random variables. 
\end{abstract}

\maketitle

%\tableofcontents

%%%%%%%%%%%%%%%%%%%%%%%%%%%%%%%%%%%%%%%%%%%%%%%%%%%%%%%%%
%%%  Introduction  %%%

\section{Introduction} \label{Introduction}

Consider $n$ parking spots placed sequentially on a one-way street. A line of $n$ cars enter the street one at a time, with each car having a preferred parking spot. The $i$th car drives to its preferred spot, $\pi_i$, and parks if the spot is available. 
If the spot is already occupied, the car parks in the first available spot after $\pi_i$. If the car is unable to find any available spots, the car exits the street without parking. 
A sequence of preferences $\pi = (\pi_1, \ldots, \pi_n)$ is a {\em parking function} if all $n$ cars are able to park. 

More precisely, let $[n] := \{1, \ldots, n\}$. A sequence $\pi = (\pi_1,\ldots, \pi_n) \in [n]^n$ is a parking function of size $n$ if and only if $|\{k : \pi_k \leq i\}| \geq i$ for all $i \in [n]$. 
This follows from the pigeonhole principle, since a parking function must have at least one coordinate with value equal to $1$, at least two coordinates with value at most $2$, and so on. 
Equivalently, $\pi = (\pi_1,\ldots, \pi_n) \in [n]^n$ is a parking function if and only if $\pi_{(i)} \leq i$ for all $i \in [n]$, where $(\pi_{(1)},\ldots, \pi_{(n)})$ is $\pi$ sorted in a weakly increasing order $\pi_{(1)} \leq \dotsb \leq \pi_{(n)}$. 

Let $\PF_n$ denote the set of parking functions of size $n$. The total number of parking functions of size $n$ is 
\bal
|\PF_n| = (n+1)^{n-1}.
\eal
An elegant proof using a circle argument is due to Pollak and can be found in \cite{FR74}. 

Parking functions were introduced by Konheim and Weiss \cite{KW66} in their study of the hash storage structure. It has since found many applications to combinatorics, probability, and computer science. 
Observe that the expression for $|\PF_n|$ is reminiscent of Cayley's formula for labeled trees. Indeed, Foata and Riordan \cite{FR74} established a bijection between parking functions and trees on $n+1$ labeled vertices. 
Connections to other combinatorial objects have since been established. For example, Stanley found connections to noncrossing set partitions \cite{Sta97} and hyperplane arrangements \cite{Sta98}, 
and Pitman and Stanley \cite{PS02} found a relation to volume polynomials of certain polytopes. 
The literature on the combinatorics of parking functions is vast and we refer the reader to Yan \cite{Yan15} for an accessible survey.  

There are several generalizations of the classical parking functions. 
In \cite{Yan01}, Yan considered ${\bm x}$-parking functions, where parking functions are associated with a vector ${\bm x}$, and constructed a bijection with rooted forests. 
Kung and Yan then introduced $[a,b]$-parking functions in \cite{KY03I} and ${\bm u}$-parking functions in \cite{KY03II}, and for both generalizations, they gave explicit formulas for moments of sums of these parking functions. 
In \cite{PS04}, Postnikov and Shapiro introduced $G$-parking functions, where $G$ is a digraph on $[n]$; the classical parking functions correspond to the case where $G = K_{n+1}$. 
Gorsky, Mazin, and Vazirani \cite{GMV16} studied rational parking functions and found connections to affine permutations and representation theory. 
Returning to the classical notion of parking on a one-way street, Ehrenborg and Happ \cite{EH16} studied parking in which cars have different sizes, and in \cite{EH18} they extended this to study parking cars behind a trailer of fixed size.  

Much of the recent combinatorial work on parking functions concerns parking completions. Suppose that $\ell$ of the $n$ spots are already occupied, denoted by $\bm{v} = (v_1,\ldots, v_\ell)$, where the entries are in increasing order, 
and that we want to find parking preferences for the remaining $n-
\ell$ cars so that they can all successfully park. The set of successful preference sequences are the {\em parking completions} for the sequence $\bm{v} = (v_1,\ldots, v_\ell)$. 
Gessel and Seo \cite{GS06} obtained a formula for parking completions where the occupied spots consist of a contiguous block starting from the first spot, $v = (1,\ldots, \ell)$. 
Diaconis and Hicks \cite{DH17} introduced the parking function shuffle to count parking completions with one spot arbitrarily occupied. 
Extending this result, Aderinan et al \cite{ABD+20} used join and split operations to count parking completions where $\ell \leq n$ occupied spots are arbitrarily occupied.  

Probabilistic questions about parking functions have also been considered, but probabilistic problems tend to be more complicated than enumeration results. 
In their study of the width of rooted labeled trees, Chassaing and Marckert \cite{CM01} discovered connections between parking functions, empirical processes, and the Brownian bridge. 
The asymptotic distribution for the cost construction for hash tables with linear probing (which is equivalent to the area statistic of parking functions) was studied by Flajolet, Poblete, and Viola \cite{FPV98} and Janson \cite{Jan01}, 
where it was shown to converge to normal, Poisson, and Airy distributions depending on the ratio between the number of cars and spots. 
 
More recently, Diaconis and Hicks \cite{DH17} studied the distribution of coordinates, descent pattern, area, and other statistics of random parking functions. 
Yao and Zeilberger \cite{YY19} used an experimental mathematics approach combined with some probability to study the area statistic. 
In \cite{KY21}, Kenyon and Yin explored links between combinatorial and probabilistic aspects of parking functions. 
Extending previous work, Yin developed the parking function multi-shuffle to obtain formulas for parking completions, moments of multiple coordinates, and all possible covariances between two coordinates for $(m,n)$-parking functions 
(where there are $m \leq n$ cars and $n$ spots) \cite{Yin21I} and $\bm{u}$-parking functions \cite{Yin21II}. 

In this paper, we continue the probabilistic study of parking functions by establishing the asymptotic distribution of the cycle counts of parking functions, partially answering a question posed by Diaconis and Hicks in \cite{DH17}.

\subsection{Main Results}

Let $C_k(\pi)$ be the number of $k$-cycles in the parking function $\pi \in \PF_n$. Our first result shows that the expected number of $k$-cycles in a random parking function is asymptotically $\frac{1}{k}$. 

\begin{theorem} \label{parkingcyclesmean}
Let $\pi \in \PF_n$ be a parking function chosen uniformly at random. Then 
\bal
E(C_k(\pi)) \sim \frac{1}{k}.
\eal
\end{theorem}

Let $\mu$ and $\nu$ be probability distributions. The {\em total variation distance} between $\mu$ and $\nu$ is
\bal
d_{TV}(\mu,\nu) := \sup_{A \subseteq \Omega} |\mu(A) - \nu(A)|,
\eal
where $\Omega$ is a measurable space. If $X$ and $Y$ are random variables with distributions $\mu$ and $\nu$, respectively, then we write $d_{TV}(X,Y)$ in place of $d_{TV}(\mu, \nu)$. 

Our main result gives an upper bound on the total variation distance between the joint distribution of cycle counts $(C_1, \ldots, C_d)$ of a random parking function and a Poisson process $(Z_1,\ldots, Z_d)$, 
where $\{Z_k\}$ are independent Poisson random variables with rate $\lambda_k = E(C_k)$. 

\begin{theorem}\label{cyclesPoissonlimit}
Let $\pi \in \PF_n$ be a parking function chosen uniformly at random. Let $C_k = C_k(\pi)$ be the number of $k$-cycles in $\pi$ and let $W = (C_1, C_2, \ldots, C_d)$. 
Let $Z = (Z_1, Z_2, \ldots, Z_d)$, where $\{Z_k\}$ are independent Poisson random variables with rate $\lambda_k = E(C_k)$. 
Then
\bal
d_{TV}(W, Z) = O\left( \frac{d^4}{n-d} \right).
\eal

Moreover if $d = o(n^{1/4})$, then the process of cycle counts converges in distribution to a process of independent Poisson random variables
\bal
(C_1, C_2, \ldots) \xrightarrow{D} (Y_1, Y_2, \ldots) 
\eal 
as $n \to \infty$, where $\{Y_k\}$ are independent Poisson random variables with rate $\frac{1}{k}$. 
\end{theorem}

The proof uses a multivariate Stein's method with exchangeable pairs. Stein's method via exchangeable pairs has previously been used to prove limit theorems in a wide range of settings. 
We refer the reader to \cite{CDM05} for an accessible survey. 
In particular, Judkovich \cite{Jud19} used this approach to prove a Poisson limit theorem for the joint distribution of cycle counts in uniformly random permutations without long cycles. 

We remark that our limit theorem parallels the result of Arratia and Tavare \cite{AT92} on the cycle structure of uniformly random permutations. 
There is a vast probabilistic literature on the cycle structure of random permutations, which include the works of Shepp and Lloyd \cite{SL66} on ordered cycle lengths 
and DeLaurentis and Pittel \cite{DP85} on a functional central limit theorem for cycle lengths with connections to Brownian motion. 
Our paper initiates a parallel study of the cycle structure in random parking functions, but further study is fully warranted. 

\subsection{Outline}

The paper is organized as follows. Section \ref{Preliminaries} introduces definitions and notation that we use throughout the paper, and gives the necessary background and relevant results that we use to prove our main theorems. 

In Section \ref{Expectedcyclecount}, we compute the exact values for the expected number of fixed points and transpositions in a random parking function, and obtain the asymptotic formula, 
Theorem \ref{parkingcyclesmean}, for the expected number of $k$-cycles, for general $k$. 

We then apply Stein's method via exchangeable pairs in Section \ref{CycleLimitTheorem} to establish the Poisson limit theorem for joint cycle counts, Theorem \ref{cyclesPoissonlimit}. 

We conclude with some final remarks and open problems in Section \ref{FinalRemarks}. 

%%%%%%%%%%%%%%%%%%%%%%%%%%%%%%%%%%%%%%%%%%%%%%%%%%%%%%%%%
%%%  Preliminaries  %%%

\section{Preliminaries} \label{Preliminaries}

In this section we introduce the definitions and notation that we use throughout the paper, and give necessary background, tools, and techniques that we use to prove our main results.  

\subsection{Definitions and Notation}

Let $X$ and $Y$ be two random variables. We let $X \overset{d}{=} Y$ denote that $X$ and $Y$ are equal in distribution, and we let $X \xrightarrow{D} Y$ denote that $X$ converges in distribution to $Y$. 

Let $a_n, b_n$ be two sequences. 
If $\lim_{n \to \infty} \frac{a_n}{b_n}$ is a nonzero constant, then $a_n \asymp b_n$ and we say that $a_n$ is of the same {\em order} as $b_n$. 
If $\lim_{n \to \infty} \frac{a_n}{b_n} = 1$, then $a_n \sim b_n$ and we say that $a_n$ is {\em asymptotic} to $b_n$. 
If there exists positive constants $c$ and $n_0$ such that $a_n \leq cb_n$ for all $n \geq n_0$, then $a_n = O(b_n)$. 
If $\lim_{n \to \infty} \frac{a_n}{b_n} = 0$, then $a_n = o(b_n)$. 

\subsection{Abel's Multinomial Theorem}

The classical binomial and multinomial theorems are generalized by {\em Abel's multinomial theorem}. We use the following version of Abel's multinomial theorem due to Yin, which was derived from Pitman \cite{Pit02}.

\begin{lemma}[\cite{Yin21I}, Theorem 3.10, Abel's Multinomial Theorem] \label{Abelmultinomial}
Let 
\bal
A_n(x_1, \ldots, x_m; p_1, \ldots, p_m) := \sum \binom{n}{\bm{s}} \prod_{j=1}^m (s_j + x_j)^{s_j + p_j},
\eal
where $\bm{s} = (s_1,\ldots, s_m)$ and $\sum_{i=1}^m s_i = n$. Then 
\bal
&A_n(x_1,\ldots, x_i, \ldots, x_j, \ldots, x_m ; p_1,\ldots, p_i, \ldots, p_j, \ldots, p_m) \\ &\qquad = A_n(x_1,\ldots, x_j, \ldots, x_i, \ldots, x_m ; p_1,\ldots, p_j, \ldots, p_i, \ldots, p_m), \\
&A_n(x_1,\ldots, x_m; p_1,\ldots, p_m) \\ &\qquad = \sum_{i=1}^m A_{n-1}(x_1,\ldots, x_{i-1}, x_{i+1}, \ldots, x_m; p_1,\ldots, p_{i-1}, p_{i+1}, \ldots, p_m), \\
&A_n(x_1,\ldots, x_m; p_1,\ldots, p_m) \\ &\qquad  = \sum_{s=0}^n \binom{n}{s} s! (x_1 + s) A_{n-s}(x_1 + s, x_2, \ldots, x_m; p_1 - 1, p_2, \ldots, p_m). 
\eal
and the following special cases hold via the above recurrences: 
\bal
A_n(x_1, \ldots, x_m; -1, \ldots, -1) &= \frac{(x_1 + \dotsb + x_m)(x_1 + \dotsb + x_m + n)^{n-1}}{(x_1x_2\dotsb x_m)}, \\
A_n(x_1, \ldots, x_m; -1, \ldots, -1, 0) &= \frac{x_m(x_1 + \dotsb + x_m + n)^n}{(x_1x_2\dotsb x_m)}.
\eal
\end{lemma}

\subsection{Parking Completions}

Suppose that spots $v_1,\ldots, v_\ell$ are already occupied. 
Recall that the {\em parking completions} for the sequence $\bm{v} = (v_1,\ldots, v_\ell)$ is the set of successful preference sequences for the remaining $n-\ell$ cars. 
The following result due to Adeniran et al is crucial for counting parking completions and computing probabilities on parking functions. 

\begin{lemma}[\cite{ABD+20}, Theorem 1.1] \label{parkingcompletion}
The number of parking completions of $\bm{v} = (v_1,\ldots, v_\ell)$ in $[n]$ is
\bal
|\PC_n(\bm{v})| = \sum_{\bm{s} \in L_n(\bm{v})} \binom{n-\ell}{\bm{s}} \prod_{i=1}^{\ell + 1} (s_i + 1)^{s_i - 1},
\eal
where 
\bal
L_n(\bm{v}) = \left\{ \bm{s} = (s_1,\ldots, s_{\ell+1}) \in \N^{\ell+1} \middle| s_1 + \dotsb + s_i \geq v_i - i \, \, \, \forall i \in [\ell], \sum_{i=1}^{\ell+1} s_i = n-\ell \right\}
\eal
\end{lemma}

As a corollary, we get a formula for parking completions when the occupied spots form a contiguous block. We use the following version due to Yin. 

\begin{lemma}[\cite{Yin21I}, Proposition 2.7] \label{parkingcompletionblock}
Let $1 \leq \ell \leq n$ and let $1 \leq k \leq n - \ell + 1$. Then
\bal
|\PC_n((\pi_1 = k, \ldots, \pi_\ell = k+\ell - 1))| = \sum_{s = 0}^{n-\ell} \binom{n-\ell}{s}(s+\ell)^{s-1} \ell (n-s-\ell + 1)^{n-s-\ell -1}.
\eal
\end{lemma}

\subsection{Stein's Method and Exchangeable Pairs}

Stein's method is a powerful technique introduced by Charles Stein which is used to bound the distance between two probability distributions. 
It has been developed for many target distributions and successfully applied to establish limit theorems in a wide range of settings. 
The main advantage of using Stein's method is that it provides an explicit error bound on the distributional approximation. 

There are many variants of Stein's method, but we use the exchangeable pairs method. The ordered pair $(W, W')$ of random variables is an {\em exchangeable pair} if $(W, W') \overset{d}{=} (W', W)$. 
We will use the following multivariate version of Stein's method for Poisson approximation via exchangeable pairs due to Chatterjee, Diaconis, and Meckes. 

\begin{theorem}[\cite{CDM05}, Proposition 10] \label{multivariateStein}
Let $W = (W_1,\ldots, W_d)$ be a random vector with values in $\N^d$ and $E(W_i) = \lambda_i < \infty$. Let $Z = (Z_1, \ldots, Z_d)$ have independent coordinates with $Z_i \sim \Poi(\lambda_i)$. 
Let $W' = (W_1',\ldots, W_d')$ be defined on the same probability space as $W$ with $(W,W')$ an exchangeable pair. Then
\bal
d_{TV}(W, Z) \leq \sum_{k=1}^d \alpha_k \left[ E|\lambda_k - c_k P(A_k)| + E|W_k - c_k P(B_k)| \right],
\eal
with $\alpha_k = \min\{1, 1.4\lambda_k^{-1/2}\}$, any choice of the $\{c_k\}$, and
\bal
A_k &= \{W_k' = W_k + 1, W_j = W_j' \text{ for $k+1 \leq j \leq d$}\}, \\
B_k &= \{W_k' = W_k - 1, W_j = W_j' \text{ for $k+1 \leq j \leq d$}\}. 
\eal
\end{theorem}

%%%%%%%%%%%%%%%%%%%%%%%%%%%%%%%%%%%%%%%%%%%%%%%%%%%%%%%%%
%%%  Expected number of k-cycles  %%%

\section{Expected Number of Cycles of a Fixed Length} \label{Expectedcyclecount}

In this section we compute the expected number of $k$-cycles in a random parking function. We consider the cases of fixed points and transpositions separately, as we are able to compute their expected values exactly. 
For general $k$, we compute the asymptotic expected number of $k$-cycles. 

\subsection{Fixed Points and Transpositions}

Let $\fp(\pi)$ and $\tc(\pi)$ be the number of fixed points and the number of transpositions, respectively, of $\pi \in \PF_n$. 
We can decompose $\fp(\pi)$ and $\tc(\pi)$ into a sum of indicator random variables as
\bal
\fp(\pi) = \sum_{i=1}^n \I_{\{\pi_i = i\}}, \qquad \tc(\pi) &= \sum_{1 \leq i < j \leq n} \I_{\{\pi_i = j, \pi_j = i\}}
\eal

\begin{proposition} \label{parkingfptcmean}
Let $\pi \in \PF_n$ be a parking function chosen uniformly at random. Then the expected number of fixed points of $\pi$ is
\bal
E(\fp(\pi)) = 1,
\eal
and the expected number of transpositions is
\bal
E(\tc(\pi)) = \frac{n}{2(n+1)}.
\eal
\end{proposition}

\begin{proof}
For fixed points, linearity of expectation gives
\bal
E(\fp(\pi)) = \sum_{i=1}^n E( \I_{\{\pi_i = i\}}) = \sum_{i=1}^n P(\pi_i = i) = \sum_{i=1}^n P(\pi_1 = i),
\eal
where the last equality follows by the symmetry of coordinates. 

By Lemma \ref{parkingcompletion}, 
\bal
&\sum_{i=1}^n |\{\pi \in \PF_n : \pi_1 = i\}| = \sum_{i=1}^n |\PC_n((i))| \\ 
&= \sum_{i=1}^n \sum_{s=0}^{n-i} \binom{n-1}{s}(s+1)^{s-1}(n-s)^{n-s-2} \\
&= \sum_{s=0}^{n-1} \binom{n-1}{s}(s+1)^{s-1}(n-s)^{n-s-2} \sum_{i=1}^{n-s} 1 \\
&= \sum_{s=0}^{n-1} \binom{n-1}{s}(s+1)^{s-1}(n-s)^{n-s-2} (n-s) \\
&= \sum_{s=0}^{n-1} \binom{n-1}{s}(s+1)^{s-1}(n-s)^{n-s-1} \\
&= A_{n-1}(1,1;-1,0) \\
&= (n+1)^{n-1}
\eal
where the last two equalities follow from Abel's multinomial theorem, Lemma \ref{Abelmultinomial}. 
Combining the above and using $|\PF_n| = (n+1)^{n-1}$ yields
\bal
E(\fp(\pi)) &= \sum_{i=1}^n P(\pi_1 = i) = \sum_{i=1}^n \frac{|\PC_n((i))|}{|\PF_n|}  = 1.
\eal

Next we consider transpositions. By linearity of expectation, 
\bal
E(\tc(\pi)) &= \sum_{1 \leq i < j \leq n} E(\I_{\{\pi_i = j, \pi_j = i\}}) = \sum_{1 \leq i < j \leq n} P(\pi_i = j, \pi_j = i) \\
&= \sum_{1 \leq i < j \leq n} P(\pi_1 = i, \pi_2 = j),
\eal
where the last equality follows by symmetry of coordinates. 

Using Lemma \ref{parkingcompletion} gives
\bal
&\sum_{1 \leq i < j \leq n} |\{\pi \in \PF_n : \pi_1 = i, \pi_2 = j\}| = \sum_{1 \leq i < j \leq n} |\PC((i,j))| \\
&= \sum_{1 \leq i < j \leq n} \sum_{\bm{s} \in L_n((i,j))} \binom{n-2}{\bm{s}} \prod_{i=1}^3 (s_i + 1)^{s_i - 1},
\eal
where 
\bal
L_n((i,j)) = \left\{ \bm{s} = (s_1, s_2, s_3) \in \N^3 \middle| s_1 \geq i-1, s_1 + s_2 \geq j-2, s_1 + s_2 + s_3 = n-2\right\}. 
\eal
This gives us the summation indices for the sum over $\bm{s}$, so that
\bal
\sum_{1 \leq i < j \leq n} |\PC(i,j)| &= \sum_{i=1}^{n-1} \sum_{j=i+1}^n \sum_{s_1 = i-1}^{n-2} \sum_{s_2 = j-2-s_1}^{n-2-s_1} \binom{n-2}{s_1, s_2, n-2-s_1-s_2} \\
&\qquad \times (s_1 + 1)^{s_1 - 1}(s_2 + 1)^{s_2 - 1}(n-1-s_1-s_2 )^{n-3-s_1-s_2} \\
&= \sum_{s_1 = 0}^{n-2} \sum_{s_2 = 0 }^{n-2-s_1} \binom{n-2}{s_1, s_2, n-2-s_1-s_2}  \\
&\qquad \times (s_1 + 1)^{s_1 - 1}(s_2 + 1)^{s_2 - 1}(n-1-s_1-s_2)^{n-3-s_1-s_2} \\
&\qquad \times \sum_{i=1}^{s_1 + 1} \sum_{j=i+1}^{s_1 + s_2 + 2} 1.
\eal
By a change of variables with $s = s_2$ and $t = n-2-s_1-s_2$ we get
\bal
\sum_{1 \leq i < j \leq n} |\PC((i,j))| &= \sum_{s = 0}^{n-2} \sum_{t = 0 }^{n-2-s} \binom{n-2}{s,t,n-2-s-t}  \\
&\qquad \times (s + 1)^{s - 1}(t + 1)^{t - 1}(n-1-s-t)^{n-3-s-t} \\
&\qquad \times \sum_{i=1}^{n-1-s-t} \sum_{j=i+1}^{n-t} 1.
\eal
Computing the inner sum yields
\bal
\sum_{i=1}^{n-1-s-t} \sum_{j=i+1}^{n-t} 1 &= \sum_{i=1}^{n-1-s-t} (n-t-i) = \frac{(n-t+s)(n-t-s-1)}{2}, 
\eal
and plugging this back gives
\bal
\sum_{1 \leq i < j \leq n} |\PC((i,j))| &= \frac{1}{2} \sum_{s = 0}^{n-2} \sum_{t = 0 }^{n-2-s} \binom{n-2}{s,t,n-2-s-t}  \\
&\qquad \times (s + 1)^{s - 1}(t + 1)^{t - 1}(n-1-s-t)^{n-2-s-t}(n-t+s).
\eal

For ease of notation, define 
\bal
f(n,s,t) := \binom{n-2}{s,t,n-2-s-t}(s + 1)^{s - 1}(t + 1)^{t - 1}(n-1-s-t)^{n-2-s-t}. 
\eal
We distribute the $(n-t+s)$ term so that the sum splits into three components
\bal
&\sum_{1 \leq i < j \leq n} |\PC((i,j))| = \frac{1}{2} \sum_{s = 0}^{n-2} \sum_{t = 0 }^{n-2-s} nf(n,s,t) - \frac{1}{2} \sum_{s = 0}^{n-2} \sum_{t = 0 }^{n-2-s} tf(n,s,t) + \frac{1}{2} \sum_{s = 0}^{n-2} \sum_{t = 0 }^{n-2-s} sf(n,s,t) \\
&\qquad = \frac{n}{2} \sum_{s = 0}^{n-2} \sum_{t = 0 }^{n-2-s} f(n,s,t) - \frac{1}{2} \sum_{t = 0}^{n-2} \sum_{s = 0 }^{n-2-t} tf(n,s,t) + \frac{1}{2} \sum_{s = 0}^{n-2} \sum_{t = 0 }^{n-2-s} sf(n,s,t) \\
&\qquad = \frac{n}{2} \sum_{s = 0}^{n-2} \sum_{t = 0 }^{n-2-s} f(n,s,t), 
\eal
where the last two sums cancel out by symmetry of $s$ and $t$. 

Finally, we use Abel's multinomial theorem, Lemma \ref{Abelmultinomial}, to get
\bal
\frac{n}{2} \sum_{s = 0}^{n-2} \sum_{t = 0 }^{n-2-s} f(n,s,t) = \frac{n}{2} A_{n-2}(1,1,1; -1,-1,0) = \frac{n(n+1)^{n-2}}{2}.
\eal

Putting everything together gives
\bal
E(\tc(\pi)) &= \sum_{1 \leq i < j \leq n} P(\pi_1 = i, \pi_2 = j) = \sum_{1 \leq i < j \leq n} \frac{|\PC((i,j))|}{|\PF_n|} \\
&= \frac{n(n+1)^{n-2}}{2(n+1)^{n-1}} = \frac{n}{2(n+1)}. \qedhere
\eal
\end{proof}

\subsection{General $k$-Cycles}

Let $C_k(\pi)$ be the number of $k$-cycles of $\pi \in \PF_n$. We can decompose $C_k(\pi)$ into a sum of indicator random variables as
\bal
C_k(\pi) = \sum_{\alpha \in A_{k}} \I_{\{\text{$\alpha$ is a $k$-cycle in $\pi$}\}},
\eal
where $A_k = \{(i_1,\ldots, i_k) : 1 \leq i_1 < \dotsb < i_k \leq n\}$. 

\begin{proof}[Proof of Theorem \ref{parkingcyclesmean}]
The proof follows similarly as in Proposition \ref{parkingfptcmean}. We will not include all the technical details, but we walk through the key ideas. 

By linearity of expectation,
\bal
E(C_k(\pi)) &= \sum_{\alpha \in A_{k}} P(\text{$\alpha$ is a $k$-cycle in $\pi$}) \\ 
&= \sum_{1 \leq i_1 < \dotsb < i_k \leq n} P(\pi_{i_1} = i_2, \ldots, \pi_{i_k} = i_1) \\ 
&= \sum_{1 \leq i_1 < \dotsb < i_k \leq n} P(\pi_{1} = i_1, \ldots, \pi_{k} = i_k) \\
&= \frac{1}{|\PF_n|} \sum_{1 \leq i_1 < \dotsb < i_k \leq n} (k-1)!|\PC_n((i_1, \ldots, i_k))|,
\eal
where we use the symmetry of coordinates in the penultimate equality. 

Applying Lemma \ref{parkingcompletion} gives
\bal
\sum_{1 \leq i_1 < \dotsb < i_k \leq n} |\PC_n((i_1, \ldots, i_k))| &= \sum_{1 \leq i_1 < \dotsb < i_k \leq n} \sum_{\bm{s} \in L_n((i_1,\ldots,i_k))} \binom{n-k}{\bm{s}} \prod_{i=1}^{k + 1} (s_i + 1)^{s_i - 1}
\eal
where 
\bal
&L_n((i_1,\ldots,i_k)) \\
&= \left\{\bm{s} = (s_1,\ldots, s_{k+1}) \in \N^{k+1} \middle| s_1 + \dotsb + s_j \geq i_j - j \, \, \, \forall j \in [k], \sum_{j=1}^{k+1} s_j = n-k \right\}
\eal

We can extract our summation bounds from $L_n((i_1,\ldots,i_k))$, and upon applying a change of variables, we obtain
\bal
&\sum_{1 \leq i_1 < \dotsb < i_k \leq n} |\PC_n((i_1, \ldots, i_k))| \\
&\qquad \qquad = \sum_{s_1 = 0}^{n-k} \sum_{s_2 = 0}^{n-k-s_1} \dotsb \sum_{s_k = 0}^{n-k-\sum_{i=1}^{k-1} s_i} \binom{n-k}{s_1,\ldots, s_k, n-k-\sum_{i=1}^k s_i} \\
&\qquad \qquad  \times \left(n-k-\sum_{i=1}^k s_i + 1\right)^{n-k-\sum_{i=1}^k s_i - 1} \prod_{i=1}^k (s_i + 1)^{s_i - 1} \\
&\qquad \qquad \times \sum_{i_1 = 1}^{n-(k-1)-\sum_{j=1}^k s_j} \left( \sum_{i_2 = i_1 + 1}^{n-(k-2)-\sum_{j=1}^{k-1} s_j} \left( \dotsb  \left( \sum_{i_k = i_{k-1} + 1} ^{n-s_1} 1\right) \dotsb \right) \right).
\eal

After some lengthy computations, the nested sum evaluates to a degree $k$ polynomial of the form
\bal
&\sum_{i_1 = 1}^{n-(k-1)-\sum_{j=1}^k s_j} \left( \sum_{i_2 = i_1 + 1}^{n-(k-2)-\sum_{j=1}^{k-1} s_j} \left( \dotsb  \left( \sum_{i_k = i_{k-1} + 1} ^{n-s_1} 1\right) \dotsb \right) \right) \\
&\qquad \qquad = \frac{(n-(k-1)-\sum_{i=1}^k s_i)n^{k-1}}{k!} + O(n^{k-1}). 
\eal

Plugging this expression back into the sum above and applying Abel's multinomial theorem, Lemma \ref{Abelmultinomial}, yields
\bal
&\sum_{1 \leq i_1 < \dotsb < i_k \leq n} |\PC_n((i_1, \ldots, i_k))| \\
&\qquad \qquad = \sum_{s_1 = 0}^{n-k} \sum_{s_2 = 0}^{n-k-s_1} \dotsb \sum_{s_k = 0}^{n-k-\sum_{i=1}^{k-1} s_i} \binom{n-k}{s_1,\ldots, s_k, n-k-\sum_{i=1}^k s_i} \\
&\qquad \qquad  \times \left(n-k-\sum_{i=1}^k s_i + 1\right)^{n-k-\sum_{i=1}^k s_i} \prod_{i=1}^k (s_i + 1)^{s_i - 1} \times \left( \frac{n^{k-1}}{k!} + O(n^{k-2}) \right) \\
& \qquad \qquad = A_{n-k}(1, 1,\ldots, 1, 1; -1, -1,\ldots, -1, 0) \left( \frac{n^{k-1}}{k!} + O(n^{k-2}) \right) \\
& \qquad \qquad = (n+1)^{n-k} \left( \frac{n^{k-1}}{k!} + O(n^{k-2}) \right). 
\eal

Therefore,
\bal
E(C_k(\pi)) &= \frac{1}{|\PF_n|} \sum_{1 \leq i_1 < \dotsb < i_k \leq n} (k-1)! |\PC_n((i_1, \ldots, i_k))| \\
&= \frac{(n+1)^{n-k}(k-1)!}{(n+1)^{n-1}} \left( \frac{n^{k-1}}{k!} + O(n^{k-2}) \right) \\
&= \frac{1}{k}\left( \frac{n}{n+1} \right)^{k-1} + O\left( \frac{1}{n} \right)
\eal
so that $E(C_k(\pi)) \sim \frac{1}{k}$, as desired. 
\end{proof}

%%%%%%%%%%%%%%%%%%%%%%%%%%%%%%%%%%%%%%%%%%%%%%%%%%%%%%%%%
%%%%%% Distributions of k-Cycles %%%%%

\section{Poisson Limit Theorem for Cycles} \label{CycleLimitTheorem}

In this section, we introduce a useful directed graph representation for parking functions and construct our exchangeable pair. We then use these constructions along with the multivariate Stein's method to prove Theorem \ref{cyclesPoissonlimit}.

\subsection{Digraph Representation of Parking Functions}

It will be useful to represent a parking function of size $n$ as {\em directed graphs}, or {\em digraphs}, on $n$ vertices with labels in $[n]$. 
For $\pi \in \PF_n$, its digraph representation is the digraph with a directed edge from $i$ to $\pi_i$, for all $i \in [n]$, so that every vertex has outdegree one. Note that fixed points are represented by a vertex with a self-loop. 
See Figure \ref{Fig1DigraphRep} for an example.
Thus digraphs of parking functions consist of connected components, with each connected component comprised of rooted trees arranged in a cycle. 
For a fixed connected component, we refer to the cycle as the {\em cycle component} and we refer to a tree as a {\em tree component}. 

Let $C_a$ denote the cycle containing the vertex $a \in [n]$ and let $L(C_a)$ be the length of $C_a$. 
Similarly, if $a$ is contained in a tree component $T_a$, let $P_a$ be the unique path between $a$ and the root of $T_a$, which lies on the cycle component, and let $L(P_a)$ be the length, or number of edges, of $P_a$. 

The following lemma is useful in approximating the probability that a vertex is contained in some cycle or path of fixed length. The bound is quite crude, but suffices for our purposes. 

%%%%%%%%%%%%%%%%%%%%%%%%%%%%
%%% FIGURE %%%
\begin{figure}
\centering
\scalebox{0.8}{
\begin{tikzpicture} [node distance = {15mm}, thick, main/.style = {draw, circle}]
\node[main] (9) {$9$};
\node[main] (6) [below right of=9] {$6$};
\node[main] (8) [above right of=6] {$8$};
\node[main] (1) [below right of=6] {$1$};
\node[main] (5) [below left of=1] {$5$};
\node[main] (7) [below of=1] {$7$};
\node[main] (2) [below right of=1] {$2$};
\node[main] (3) [above right of=2] {$3$};
\node[main] (11) [below right of=2] {$11$};
\node[main] (12) [above right of=3] {$12$};
\node[main] (10) [above right of=12] {$10$};
\node[main] (4) [above left of=10] {$4$};
\draw[->] (6) to [out=180, in=270] (9); 
\draw[->] (9) to [out=90, in=90] (8); 
\draw[->] (8) to [out=270, in=0] (6); 
\draw[->] (1) -- (6); 
\draw[->] (5) -- (1);
\draw[->] (7) -- (1);
\draw[->] (2) -- (1); 
\draw[->] (11) -- (2); 
\draw[->] (3) -- (2); 
\draw[->] (12) -- (10); 
\draw[->] (10) -- (4); 
\draw[->] (4) to [out=180, in=270, looseness = 8] (4); 
\end{tikzpicture}
}
\caption{The digraph representation of the parking function $\pi =  (6 \, 1 \, 2 \, 4 \, 1 \, 9 \, 1 \, 6 \, 8 \, 4 \, 2 \, 10) \in \PF_{12}$. It consists of two components, where each component consists of a tree component attached to a cycle component.}
\label{Fig1DigraphRep}
\end{figure}
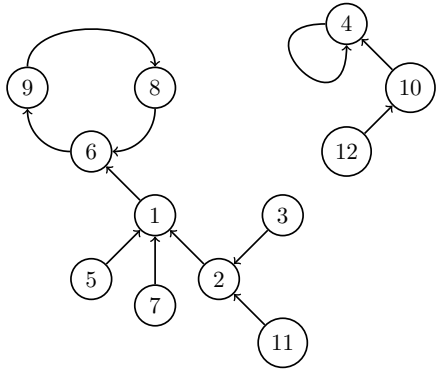
%%%%%%%%%%%%%%%%%

\begin{lemma} \label{cycleprobUB}
Fix $a \in [n]$ and let $\pi \in \PF_n$ be a parking function chosen uniformly at random. Then
\bal
P(L(C_a) = k) \leq \frac{k+1}{n+1}, 
\eal
for all $k \in [n]$. The same upper bound holds for $P(L(P_a) = k)$. 
\end{lemma}

\begin{proof}
For fixed $a \in [n]$, there are $\binom{n-1}{k-1}$ ways to pick the remaining values to be in cycle $C_a$, and there are $(k-1)!$ distinct cycles that can be formed from these $k$ values. 
Suppose $i_1 < \dotsb < i_k$ are the values in cycle $C_a$ placed in increasing order, where $a$ is one of the $i_j$'s. There are $|\PC_n((i_1, \ldots, i_k))|$ ways to complete $\pi$ so that it is a parking function. 

Observe that the value of $|\PC_n((i_1, \ldots, i_k))|$ differs, depending on the values of $i_1, \ldots, i_k$. However, note that $|\PC_n((i_1, \ldots, i_k))|$ will be maximized when $i_j = j$ for $1 \leq j \leq k$. 
That is, if cycle $C_a = (1,\ldots,k)$. To see why this is true, observe that the sum in Lemma \ref{parkingcompletion} decreases as each $i_j$ increases past $j$, since there will be fewer resulting summands. It follows that
\bal
|\PC_n((i_1, \ldots, i_k))| \leq |\PC_n((1, \ldots, k))|
\eal
for any choice of $i_1, \ldots, i_k$. 

By Lemma \ref{parkingcompletionblock},
\bal
|\PC_n((1, \ldots, k))| &= \sum_{s=0}^{n-k} \binom{n-k}{s}(s+1)^{s-1} k (n-s)^{n-s-k -1} \\
&= k A_{n-k}(1,k; -1, -1) \\
&= (k+1)(n+1)^{n-k-1},
\eal
where last equality follows from Abel's multinomial theorem, Lemma \ref{Abelmultinomial}. Thus
\bal
P((1,\ldots, k) \text{ is a $k$-cycle}) &= \frac{(k-1)! |\PC_n((1,\ldots, k))|}{|\PF_n|} \\
&= \frac{(k-1)!(k+1)(n+1)^{n-k-1}}{(n+1)^{n-1}} \\
&= \frac{(k-1)!(k+1)}{(n+1)^k}.
\eal 

Therefore 
\bal
P(L(C_a) = k) &\leq \binom{n-1}{k-1}P((1,\ldots, k) \text{ is a $k$-cycle}) = \frac{\binom{n-1}{k-1} (k-1)!(k+1)}{(n+1)^k} \\
&= \frac{(n-1)\dotsb (n-k+1) (k+1)}{(n+1)^k} \leq \frac{(n+1)^{k-1} (k+1)}{(n+1)^k} \\
&= \frac{k+1}{n+1}.
\eal

Next we turn to the bound for $P(L(P_a) = k)$. Recall that $P_a$ is the unique path between $a$ and the root of of the tree component $T_a$. Hence if $L(P_a) = k$, then there are $k+1$ vertices on path $P_a$ including $a$ and the root. 
There are $\binom{n-2}{k-1}$ ways to pick the vertices between $a$ and the root, and there are $(k-1)!$ ways to permute these vertices. The rest of the argument follows in the same way as the cycle case. 
\end{proof}

%%% Upper Bound Computations %%%

\subsection{Upper Bound on the Total Variation Distance}

We begin by constructing an exchangeable pair. 

Pick a transposition $\tau = (a \, b) \in S_n$ uniformly at random. Let $\pi' = \pi \circ \tau$, so that the values at positions $a$ and $b$ are transposed. 
That is, $\pi_a$ and $\pi_b$ are transposed, so that $a$ is sent to $\pi_b$ and $b$ is sent to $\pi_a$ in $\pi'$. Set $W' = W(\pi')$. Then by construction $(W, W')$ is an exchangeable pair. 

Define the events
\bal
A_k &= \{W_k' = W_k + 1, W_j = W_j' \text{ for $k+1 \leq j \leq d$}\}, \\
B_k &= \{W_k' = W_k - 1, W_j = W_j' \text{ for $k+1 \leq j \leq d$}\}. 
\eal

In order to apply the multivariate Stein's method via exchangeable pairs, we will need the following two lemmas. 

\begin{lemma} \label{SteinAkUB} %%% FIX THIS!!! Should be c_k = n/6k. Also factor of 2 in last summand %%%
Let $(W, W')$ be the exchangeable pair defined above. If $\lambda_k = E(W_k)$ and $c_k = \frac{n}{6k}$, then
\bal
E|\lambda_k - c_k P(A_k)| &\leq \frac{2d^2 + 2dk + 2d + 8k^2 + 4k + 1}{k(n-1)} + \frac{k^2 + 6k}{n-k+1} + O(n^{-1}).
\eal
\end{lemma}

\begin{proof}
Let $\tau = (a \, b)$ be the transposition in the exchangeable pairs construction for $W'$. We count the transpositions such that $A_k$ occurs. There are four cases. 

\begin{enumerate}

\item If $a$ and $b$ are in different cycles $C_a$ and $C_b$, respectively, then $W_k' = W_k+1$ if $L(C_a) + L(C_b) = k$, since swapping $\pi_a$ and $\pi_b$ strings together the two cycles into a single cycle of length $k$. 

\item Suppose $a$ and $b$ are in the same cycle component $C_a$. %Denote this event by $F_1$. 
Then $\tau$ breaks $C_a$ into two smaller cycles, one of which has length $k$ and another of length $L(C_a) - k$. Thus if $L(C_a) > k$, then there are two such transpositions. 
There are three subcases where $A_k$ does not occur. 

If $L(C_a) = 2k$, then $C_a$ splits into two cycles of length $k$ so that $W_k' = W_k + 2$. 
Next, if $L(C_a) \in \{k+1, k+2, \ldots, d\}$, then $W_{L(C_a)}' = W_{L(C_a)} - 1$.
Finally, if $L(C_a) \in \{2k+1, 2k+2, \ldots, d+k\}$, then $W_{L(C_a) - k}' = W_{L(C_a) - k} + 1$. 
Therefore we must have that $d < L(C_a) < 2k$ or $L(C_a) > d+k$. 

\item Suppose $a$ and $b$ are in the same connected component such that $a$ is in the tree component and $b$ is in the cycle component $C_b$. %Denote this event by $F_2$. 
If $L(P_a) + L(C_b) > k$, then there is exactly one transposition that breaks $C_b$ into a smaller cycle of length $k$ and a path of length $L(P_a) + L(C_b) - k$. 
There are two subcases where $A_k$ does not occur. 

If $L(C_b) = k$, then $W_k' = W_k$ since a $k$-cycle is created but the $k$-cycle $C_b$ is destroyed. Second, if $L(C_b) \in \{k+1, k + 2, \ldots, d\}$, then $W_{L(C_b)}' = W_{L(C_b)} - 1$. Therefore we must have that $L(C_b) < k$ or $L(C_b) > d$. 

\item Finally suppose $a$ and $b$ are both in the same tree component such that $b$ lies on the unique path from $a$ to the root. %Denote this event by $F_3$. 
If $L(P_a) > k$, then there is exactly one transposition that breaks $P_a$ into a cycle of length $k$ and a path of length $L(P_a) - k$. 
\end{enumerate}

Combining the cases above gives
\bal
P(A_k) = &\frac{1}{n(n-1)} \sum_{a=1}^n \sum_{b \neq a} \I_{\{L(C_a) + L(C_b) = k\}} \I_{\{C_a \neq C_b\}} \\
&+ \frac{2}{n(n-1)} \sum_{a=1}^n (\I_{\{d < L(C_a) < 2k\}} + \I_{\{L(C_a) > d+k\}}) \\
&+ \frac{2}{n(n-1)} \sum_{a=1}^n \I_{\{L(P_a) + L(C_b) > k\}} ( \I_{\{L(C_b) < k\}} + \I_{\{L(C_b) > d\}} ) \\
&+ \frac{2}{n(n-1)} \sum_{a=1}^n \I_{\{L(P_a) > k\}}.
\eal

Rewriting using $\I_{\{L(C_a) > d+k\}} = 1 - \I_{\{L(C_a) \leq d+k\}}$, $\I_{\{L(P_a) + L(C_b) > k\}} = 1 - \I_{\{L(P_a) + L(C_b) \leq k\}}$, $\I_{\{L(C_b) > d\}} = 1 - \I_{\{L(C_b) \leq d\}}$, and $\I_{\{L(P_a) > k\}} = 1 -  \I_{\{L(P_a) \leq k\}}$ yields, 
upon simplification,
\bal
P(A_k) = &\frac{6}{n-1} + \frac{1}{n(n-1)} \sum_{a=1}^n \sum_{b \neq a} \I_{\{L(C_a) + L(C_b) = k\}} \I_{\{C_a \neq C_b\}} \\
&+ \frac{2}{n(n-1)} \sum_{a=1}^n (\I_{\{d < L(C_a) < 2k\}} - \I_{\{L(C_a) \leq d+k\}}) \\
&+ \frac{1}{n(n-1)} \sum_{a=1}^n (\I_{\{L(C_b) < k\}} - \I_{\{L(C_b) \leq d\}} - \I_{\{L(P_a) + L(C_b) \leq k\}}) \\
&- \frac{1}{n(n-1)} \sum_{a=1}^n \I_{\{L(P_a) \leq k\}}.
\eal

Using the fact that $\lambda_k = \frac{1}{k} + O(n^{-1})$ from the proof of Theorem \ref{parkingcyclesmean}, $c_k = \frac{n}{6k}$, and the triangle inequality, we get
\bal
&E|\lambda_k - c_k P(A_k)| \\
&\leq \frac{1}{k(n-1)} + O(n^{-1}) + \frac{1}{6k(n-1)} \sum_{a=1}^n \sum_{b \neq a} E(\I_{\{L(C_a) + L(C_b) = k\}} \I_{\{C_a \neq C_b\}}) \\
&+ \frac{1}{3k(n-1)} \sum_{a=1}^n E(\I_{\{d < L(C_a) < 2k\}} + \I_{\{L(C_a) \leq d+k\}}) \\
&+ \frac{1}{3k(n-1)} \sum_{a=1}^n E( \I_{\{L(C_b) < k\}} + \I_{\{C_b \leq d\}} + \I_{\{L(P_a) + L(C_b) \leq k\}}) \\
&+ \frac{1}{3k(n-1)} \sum_{a=1}^n E(\I_{\{L(P_a) \leq k\}}).
\eal

We bound the expectations in the four summands above. Conditioning on the length of $C_a$ and applying Lemma \ref{cycleprobUB}, the first summand is
\bal
&E(\I_{\{L(C_a) + L(C_b) = k, C_a \neq C_b\}}) = \sum_{j=1}^{k-1} E(\I_{\{C_a \neq C_b\}}\I_{\{L(C_a) = j\}}\I_{\{L(C_b) = k-j\}}) \\
&\qquad \leq \sum_{j=1}^{k-1} P(L(C_b) = k-j \mid C_a \neq C_b, L(C_a) = j)P(L(C_a) = j) \\
&\qquad \leq \sum_{j=1}^{k-1} \left(\frac{k-j+1}{n-k+1}\right) \left( \frac{j+1}{n+1}\right) \\
&\qquad = \frac{k^3 + 6k^2 - k - 6}{6(n+1)(n-k+1)}.
\eal
By the union bound and Lemma \ref{cycleprobUB}, the second summand is
\bal
E(\I_{\{d < L(C_a) < 2k\}} + \I_{\{L(C_a) \leq d+k\}}) &\leq 2kP(L(C_a) = 2k) + (d+k)P(L(C_a) = d+k) \\
&= \frac{2k(2k+1) + (d+k)(d+k+1)}{n+1}. 
\eal
Similarly, by Lemma \ref{cycleprobUB}, the third summand is 
\bal
&E( \I_{\{L(C_b) < k\}} + \I_{\{C_b \leq d\}} + \I_{\{L(P_a) + L(C_b) \leq k\}}) \\
&\qquad \leq (k-1)P(L(C_b) = k-1) + dP(L(C_b) = d) + kP(L(P_a) + L(C_b) = k) \\
&\qquad \leq \frac{k(k-1) + d(d+1) + k(k+1)}{n+1}.
\eal
Finally, by Lemma \ref{cycleprobUB}, the fourth summand is
\bal
E( \I_{\{L(P_a) \leq k\}}) \leq kP(L(P_a) = k) \leq \frac{k(k+1)}{n+1}.
\eal

Therefore combining the above yields the upper bound
\bal
&E|\lambda_k - c_k P(A_k)| \\
&\qquad \leq \frac{1}{k(n-1)} + O(n^{-1}) + \frac{1}{6k(n-1)} \sum_{a=1}^n \sum_{b \neq a} \frac{k^3 + 6k^2 - k - 6}{6(n+1)(n-k+1)} \\
&\qquad + \frac{1}{3k(n-1)} \sum_{a=1}^n \frac{2k(2k+1) + (d+k)(d+k+1)}{n+1} \\
&\qquad + \frac{1}{3k(n-1)} \sum_{a=1}^n \frac{k(k-1) + d(d+1) + k(k+1)}{n+1} \\
&\qquad + \frac{1}{3k(n-1)} \sum_{a=1}^n \frac{k(k+1)}{n+1} \\
&\quad \leq \frac{1}{k(n-1)} + O(n^{-1}) + \frac{k^2 + 6k}{n-k+1} \\
&\qquad + \frac{d^2 + 2dk + d + 5k^2 + 3k}{k(n-1)}  + \frac{d^2 + d + 2k^2}{k(n-1)} + \frac{k+1}{n-1} \\ 
&\quad = \frac{2d^2 + 2dk + 2d + 8k^2 + 4k + 1}{k(n-1)} + \frac{k^2 + 6k}{n-k+1} + O(n^{-1}). \qedhere
\eal
\end{proof}

\begin{lemma} \label{SteinBkUB}
Let $(W, W')$ be the exchangeable pair defined above. If $c_k = \frac{n}{6k}$, then
\bal
E|W_k - c_k P(B_k)| &\leq \frac{k+1}{n-1} + \frac{dk^3 + 6k^3 - 2dk^2 + 3d^2k - 6k + 3d^2 + 3d}{k(n-k+1)}. 
\eal
\end{lemma}

\begin{proof}
Let $\tau = (a \, b)$ be the transposition in the exchangeable pairs construction for $W'$. We count the transpositions such that $B_k$ occurs. 
Observe that if $a \in C_a$ where $C_a$ is a $k$-cycle, then any transposition $\tau$ such that $b \neq a$ will break the cycle $C_a$. There are four cases.

\begin{enumerate}

\item If $a$ and $b$ are in different cycles $C_a$ and $C_b$, then $\tau$ breaks $C_a$ and $C_b$ and strings them together to form a cycle of length $L(C_b) + k$. Thus we must have that either $L(C_b) > d$, or $L(C_b) < k$ and $L(C_b) + k > d$. 

\item Suppose $a$ and $b$ are both in the same cycle component $C_a$. % Denote this event by $F_1$. 
Then $\tau$ breaks $C_a$ into two smaller cycles. For fixed $a$, there are $k-1$ choices for $b$. 

\item Suppose $a$ and $b$ are in different components, with $a \in C_a$ and $b$ in a tree component. Denote this event by $G_1$. Then $\tau$ breaks $C_a$ and $P_b$, and creates a path of length $L(P_b) + L(C_a)$. 

\item Suppose $a$ and $b$ are in the same component, with $a \in C_a$ and $b$ on the tree component. %Denote this event by $F_2$. 
Then $\tau$ creates a cycle whose length lies in the interval $[L(P_b), L(P_b)+L(C_a)-1]$ and a path. The created cycle must have length greater than $d$ or less than $k$. There are three cases. 

If $L(P_b) > d$, then $\tau$ creates a cycle of length greater than $d$. If $L(P_b) = d-j+1$ for $1 \leq j \leq k-1$, then there are exactly $(k-j)$ transpositions $\tau$ which create a cycle of length greater than $d$. 
Finally if $L(P_b) = j$ for $1 \leq j \leq k-1$, then there are exactly $(k-j)$ transpositions $\tau$ which creates a cycle of length less than $k$. 

\end{enumerate}

Combining the cases above gives
\bal
P(B_k) &= \frac{2}{n(n-1)} \sum_{a = 1}^n \sum_{b \neq a} \I_{\{L(C_a) = k\}} \I_{\{C_a \neq C_b\}} [\I_{\{L(C_b) > d\}} + \I_{\{L(C_b) < k\}}\I_{\{L(C_b) > d - k\}}] \\
&\quad + \frac{k-1}{n(n-1)} \sum_{a = 1}^n \I_{\{L(C_a) = k\}} + \frac{2}{n(n-1)} \sum_{a = 1}^n \sum_{b \neq a} \I_{\{L(C_a) = k\}} \I_{G_1} \\
&\quad + \frac{2}{n(n-1)} \sum_{a = 1}^n \sum_{a \neq b} \I_{\{L(C_a) = k\}} [\I_{\{L(P_b) > d\}} \\ 
&\qquad + \sum_{j=1}^{k-1} (k-j)(\I_{\{L(P_b) = d-j+1\}} + \I_{\{L(P_b) = j\}})], 
\eal 

Rewriting using $\I_{\{L(C_b) > d\}} = 1 - \I_{\{L(C_b) \leq d\}}$, $\I_{\{L(C_b) > d - k\}} = 1 - \I_{\{L(C_b) \leq d - k\}}$, $\I_{\{L(P_b) > d\}} = 1 - \I_{\{L(P_b) \leq d\}}$, and simplifying gives 
\bal
P(B_k) &= \frac{6}{n} \sum_{a = 1}^n \I_{\{L(C_a) = k\}} + \frac{2}{n(n-1)} \sum_{a = 1}^n \sum_{b \neq a} \I_{\{L(C_a) = k\}} \I_{\{C_a \neq C_b\}} [\I_{\{L(C_b) < k\}} \\ 
&\qquad - \I_{\{L(C_b) < k\}}\I_{\{L(C_b) \leq d - k\}} - \I_{\{L(C_b) \leq d\}}] \\ 
&\quad + \frac{k-1}{n(n-1)} \sum_{a = 1}^n \I_{\{L(C_a) = k\}} + \frac{2}{n(n-1)} \sum_{a = 1}^n \sum_{a \neq b} \I_{\{L(C_a) = k\}} [-\I_{\{L(P_b) \leq d\}} \\ 
&\qquad + \sum_{j=1}^{k-1} (k-j)(\I_{\{L(P_b) = d-j+1\}} + \I_{\{L(P_b) = j\}})], 
\eal

Observe that we may write the number of $k$-cycles as a sum of indicators as $W_k = \frac{1}{k} \sum_{a=1}^n \I_{\{L(C_a) = k\}}$. 
To see this, note that every member of a cycle of length $k$ contributes an indicator variable of value $1$ to the sum, so that each $k$-cycle is counted $k$ times. 
Thus dividing the sum by $k$ yields the correct number of $k$-cycles. 

Using this representation of $W_k$ along with $c_k = \frac{n}{6k}$ and the triangle inequality gives 
\bal
&E|W_k - c_k P(B_k)| \\
&\leq \frac{1}{3k(n-1)} \sum_{a = 1}^n \sum_{b \neq a} E(\I_{\{L(C_a) = k\}} \I_{\{C_a \neq C_b\}} [\I_{\{L(C_b) < k\}} \\
&\qquad + \I_{\{L(C_b) < k\}}\I_{\{L(C_b) \leq d - k\}} + \I_{\{L(C_b) \leq d\}}]) \\ 
&\quad + \frac{k-1}{6k(n-1)} \sum_{a=1}^n E(\I_{\{L(C_a) = k\}}) + \frac{1}{3k(n-1)} \sum_{a = 1}^n \sum_{a \neq b} E(\I_{\{L(C_a) = k\}} [\I_{\{L(P_b) \leq d\}} \\ 
&\qquad + \sum_{j=1}^{k-1} (k-j)(\I_{\{L(P_b) = d-j+1\}} + \I_{\{L(P_b) = j\}})])
\eal

We bound the expected values in the summands. The computations are similar to those in the proof of Lemma \ref{SteinAkUB} so we omit some technical details. 
By the union bound and Lemma \ref{cycleprobUB}, the first summand is
\bal
&E(\I_{\{L(C_a) = k\}} \I_{\{C_a \neq C_b\}} [\I_{\{L(C_b) < k\}} + \I_{\{L(C_b) < k\}}\I_{\{L(C_b) \leq d - k\}} + \I_{\{L(C_b) \leq d\}}]) \\ 
&\quad \leq P(L(C_a) = k)[P(L(C_b) < k \mid L(C_a) = k, C_a \neq C_b) \\ 
&\qquad + P(L(C_b) < k, L(C_b) \leq d - k \mid L(C_a) = k, C_a \neq C_b) \\
&\qquad + P(L(C_b) \leq d \mid L(C_a) = k, C_a \neq C_b)] \\
&\quad \leq \frac{k+1}{n+1}\left( \frac{2k(k-1) + (d-k)(d-k+1) + d(d+1)}{n-k+1} \right) \\
&\quad \leq \frac{2d^2k + 2d^2 - 2dk^2 + 2d + 3k^3 - 3k}{(n+1)(n-k+1)}.
\eal 
By Lemma \ref{cycleprobUB}, the second summand is
\bal
E(\I_{\{L(C_a) = k\}}) = P(L(C_a) = k) \leq \frac{k+1}{n+1}.
\eal
Finally, the third summand is 
\bal
&E(\I_{\{L(C_a) = k\}} [\I_{\{L(P_b) \leq d\}} + \sum_{j=1}^{k-1} (k-j)(\I_{\{L(P_b) = d-j+1\}} + \I_{\{L(P_b) = j\}})]) \\
&\leq \frac{k+1}{n+1}\left( \frac{d(d+1)}{n-k+1} + \sum_{j=1}^{k-1} (k-j) \left( \frac{d-j+2 + j + 1}{n-k+1} \right) \right) \\
&\leq \frac{d^2k + d^2 + dk^3 + d + 3k^3 - 3k}{(n+1)(n-k+1)}.
\eal

Combining the above gives us the upper bound 
\bal
&E|W_k - c_k P(B_k)| \\
&\leq \frac{1}{3k(n-1)} \sum_{a = 1}^n \sum_{b \neq a} \frac{2d^2k + 2d^2 - 2dk^2 + 2d + 3k^3 - 3k}{(n+1)(n-k+1)} \\
&\quad + \frac{k-1}{6k(n-1)} \sum_{a=1}^n \frac{k+1}{n+1} \\
&\quad + \frac{1}{3k(n-1)} \sum_{a = 1}^n \sum_{b \neq a} \frac{d^2k + d^2 + dk^3 + d + 3k^3 - 3k}{(n+1)(n-k+1)} \\
&\leq \frac{2d^2k + 2d^2 - 2dk^2 + 2d + 3k^3 - 3k}{k(n-k+1)} + \frac{k+1}{n-1} \\
&\quad + \frac{d^2k + d^2 + dk^3 + d + 3k^3 - 3k}{k(n-k+1)} \\
&= \frac{k+1}{n-1} + \frac{dk^3 + 6k^3 - 2dk^2 + 3d^2k - 6k + 3d^2 + 3d}{k(n-k+1)}. \qedhere
\eal
\end{proof}

Using the two lemmas above with the multivariate version of Stein's method via exchangeable pairs, we can now prove our Poisson limit theorem for cycle counts. 

\begin{proof}[Proof of Theorem \ref{cyclesPoissonlimit}]
By Theorem \ref{multivariateStein} and Lemmas \ref{SteinAkUB} and \ref{SteinBkUB},
\bal
&d_{TV}(W, Z) \leq \sum_{k=1}^d (E|\lambda_k - c_k P(A_k)| + E|W_k - c_k P(B_k)| ) \\
&\leq \sum_{k=1}^d \left(\frac{2d^2 + 2dk + 2d + 8k^2 + 4k + 1}{k(n-1)} + \frac{k^2 + 6k}{n-k+1} + O(n^{-1})\right) \\
&\qquad + \sum_{k=1}^d \left(\frac{k+1}{n-1} + \frac{dk^3 + 6k^3 - 2dk^2 + 3d^2k - 6k + 3d^2 + 3d}{k(n-k+1)}\right) \\ 
&\leq \frac{1}{n-d} \sum_{k=1}^d \left(\frac{dk^3 + 7k^3 - dk^2 + 15k^2 + 3d^2k + 2dk - k + 5d^2 + 5d + 1}{k} \right) + O(d/n) \\
&= \frac{1}{n-d} \left( (5d^2 + 5d + 1)H_d + \frac{d^4 + 16d^3 + 38d^2 + 23d}{3} \right) + O(d/n), 
\eal
where $H_d$ is the $d$th harmonic number. Using the fact that $H_d \leq \log d + 1$ and simplifying, we obtain an upper bound of
\bal
d_{TV}(W, Z) &\leq \frac{d^4 + 6d^3 + 18d^2 + 13d + 1 + (5d^2 + 5d + 1) \log d}{n-d} + O(d/n) \\
&= O\left( \frac{d^4}{n-d} \right). 
\eal

For $d = o(n^{1/4})$, we have that $d_{TV}(W, Z) \to 0$ as $n \to \infty$. Let $Y = (Y_1,\ldots, Y_d)$, where $\{Y_k\}$ are Poisson random variables with rate $\frac{1}{k}$. 
Since $\lambda_k \to \frac{1}{k}$ as $n \to \infty$ by Theorem \ref{parkingcyclesmean}, we have that $d_{TV}(Z, Y) \to 0$ as $n \to \infty$. By the triangle inequality, $d_{TV}(W, Y) \to 0$ as $n \to \infty$. 
It follows that for all fixed $d$, $(C_1, \ldots, C_d) \xrightarrow{D} (Y_1, \ldots, Y_d)$ as $n \to \infty$. Therefore 
\bal
(C_1, C_2, \ldots) \xrightarrow{D} (Y_1, Y_2, \ldots)
\eal
as $n \to \infty$. 
\end{proof}

We remark that if $d$ is a fixed constant, the total variation distance upper bound is $O(n^{-1})$, which is the optimal convergence rate for Poisson approximation. 

%%%%%%%%%%%%%%%%%%%%%%%%%%%%%%%%%%%%%%%%%%%%%%%%%%%%%%%%%
%%%%%% Final Remarks %%%%%

\section{Final Remarks} \label{FinalRemarks}

\subsection{} Parking functions are a subset of a more general class of functions, $\Fc_n$, called {\em random mappings}, which are functions $f : [n] \to [n]$ from the set $[n]$ to itself. 
Random mappings are extensively used in computer science and computational mathematics, for example in random number generators, cycle detection, and integer factorization. 

There is an extensive literature on the probabilistic properties of random mappings and we only cite a few here. 
Harris \cite{Har60} initiated the classical theory of random mappings and studied various probability distributions related to random mappings. 
In \cite{Han89}, Hansen proved a functional central limit theorem for the component structure of the graph representation of random mappings. 
Subsequently, Flajolet and Odlyzko \cite{FO90} studied various statistics on random mappings and used analytic combinatorics to prove limit theorems on these statistics. 
In \cite{AP94}, Aldous and Pitman found connections between features on random mappings and features on Brownian bridge. A parallel probabilistic study on parking functions is fully warranted. One direction is to study various statistics on random parking functions and find their limiting distributions. 

\subsection{} It would be interesting to see if there is a generating function approach to prove the Poisson limit theorem for cycle counts. 
Generating functions combined with analytic combinatorics and singularity analysis have been widely used to compute moments for statistics in various random combinatorial structures, 
as well as in proving limit theorems. 
In particular, it would be interesting to place parking functions into the theoretical framework of logarithmic combinatorial assemblies, introduced by Arratia, Stark, and Tavare in \cite{AST95}. 
Some examples of assemblies include permutations, mappings, and set partitions. 

\subsection{} In \cite{Yin21I} and \cite{Yin21II}, Yin initiated the probabilistic study of $(m,n)$-parking functions and $\bm{u}$-parking functions, respectively, and in particular, obtained explicit formulas for their parking completions. 
It should be tractable to follow our approach and use Stein's method via exchangeable pairs to obtain limit theorems, with convergence rates, for the distribution of cycle counts in these more general models.

\subsection{} A generating function approach was used in \cite{SL66} by Shepp and Lloyd to show that the total number of cycles in a uniformly random permutation is asymptotically normal with mean and variance $\log n$. 
Similarly, Flajolet and Odlyzko \cite{FO90} used generating functions to show that the total number of cycles in a uniformly random mapping is asymptotically normal with mean and variance $\frac{1}{2} \log n$. 
Note that asymptotically, random mappings have about half as many total cycles as random permutations. 

Let $K_n(\pi) = \sum_{k=1}^n C_k(\pi)$ be the total number of cycles of a uniformly random parking function $\pi \in \PF_n$. By Theorem \ref{parkingcyclesmean},  
\bal
E(K_n(\pi)) = \sum_{k=1}^n E(C_k(\pi)) \asymp \sum_{k=1}^n \frac{1}{k} = H_n
\eal
where $H_n$ is the $n$th harmonic number. Note that although $E(C_k(\pi)) \sim \frac{1}{k}$, the implied error terms are not uniform in $k$. Thus we are only able to get that $E(K_n(\pi)) \asymp \log n$. 
A more careful analysis should give the correct constant in front of the log term. 
It should be tractable to use Stein's method for normal approximation to show that $K_n$ is asymptotically normal along with a rate of convergence. 

%%%%%%%%%%%%%%%%%%%%%%%%%%%%%%%%%%%%%%%%%%%%%%%%%%%%%%%%%
%%%%%% Acknowledgements %%%%%

\section*{Acknowledgments} 
We thank Jason Fulman for suggesting this problem and Peter Kagey for a helpful discussion. We also thank an anonymous referee for many helpful comments and suggestions. 

%%%%%%%%%%%%%%%%%%%%%%%%%%%%%%%%%%%%%%%%%%%%%%%%%%%%%%%%%
%%%%%% References %%%%%

\Address

\end{document}